\documentclass[psamsfonts]{amsart}

\usepackage[utf8]{inputenc}
\usepackage[dvips]{graphicx}
\usepackage{pstricks}

\usepackage{hyperref}
\input xy
\xyoption{all}

\usepackage{boxedminipage}
\usepackage{color}

\newcommand{\ZZ}{{\mathbb Z}}
\newcommand{\NN}{{\mathbb N}}
\newcommand{\PP}{{\mathbb P}}
\newcommand{\KK}{{\mathbb K}}
\newcommand{\CC}{{\mathbb C}}
\newcommand{\RR}{{\mathbb R}}

\newcommand{\TT}{{\mathbb T}}

\newcommand{\E}{{\mathcal E}}

\DeclareMathOperator{\Log}{{\textup{Log}}}
\DeclareMathOperator{\Aff}{{\textup{Aff}}}

\newtheorem{thm}{Theorem}[section]

\newtheorem{defi}[thm]{Definition}

\newtheorem{prop}[thm]{Proposition}

\newtheorem{lemma}[thm]{Lemma}

\newtheorem{cor}[thm]{Corollary}

         {\theoremstyle{definition}

\newtheorem{rem}[thm]{Remark}}

         {\theoremstyle{definition}

\newtheorem{exa}[thm]{Example}}

\newcommand{\fonction}[5]
	{ \begin{array}[t]{rccc}
 #1~: 	&	#2	&	\rightarrow	&	#3	\\
	&	#4	&	\mapsto		&	#5
	\end{array}}

\begin{document}

\title{On the Total Curvature of Tropical Hypersurfaces}

\date{\today}

\author[Bertrand]{Benoît Bertrand}

\address{Université Paul Sabatier, Institut Mathématiques de Toulouse, 118 route de Narbonne, F-31062 Toulouse Cedex 9, France}

\email{benoit.bertrand@math.univ-toulouse.fr}

\author[López de Medrano]{Lucía López de Medrano}

\address{Unidad Cuernavaca del Instituto de Matemáticas,Universidad Nacional
Autonoma de México. Cuernavaca, México}

\email{lucia@matcuer.unam.mx}

\author[Risler]{Jean-Jacques Risler}

\address{Université Pierre et Marie Curie,  Paris 6, 
4 place Jussieu, 75005 Paris, France}

\email{risler@math.jussieu.fr}

\copyrightinfo{2013}{Benoît BERTRAND, Lucía LÓPEZ DE MEDRANO, Jean-Jacques RISLER}

\subjclass[2010]{Primary 14T05, 14P05; Secondary 52B20}

\keywords{Logarithmic Curvature, Amoebas, Real Tropical Hypersurfaces, Polyhedral Hypersurfaces}

\thanks{The authors want to thank the Laboratorio Internacional Solomon Lefschetz (LAISLA), associated to the CNRS (France) and CONACYT (Mexico). Its support was crucial in the development of this project. Bertrand and Risler were partially supported by the ANR-09-BLAN-0039-02. López de Medrano was partially supported by UNAM-PAPIIT IN-117110 and UNAM-PAPIIT IN-108111. Risler was also partially supported by UNAM-PAPIIT IN-108112.}

\begin{abstract}
This paper studies the curvatures of amoebas and real amoebas (i.e. essentially logarithmic curvatures of the complex and real parts of a real algebraic hypersurface) and of tropical and real tropical hypersurfaces. 
 If $V$ is a tropical hypersurface defined over the field  of real Puiseux series, it has a real part $\RR V$ which is a polyhedral complex. 
We define the total curvature of $V$ (resp. $\RR V$) by using the total curvature of Amoebas and passing to the limit.
 We also define the  ``polyhedral total curvature'' of the real part $\RR V$ of a generic tropical hypersurface. 
 The main results we prove about these notions are the following:
\begin{enumerate}
\item The fact that 
 the total curvature and the polyhedral total curvature coincide for real non-singular tropical hypersurfaces. 
\item A universal inequality between the total curvatures of $V$ and $\RR V$ and another between the logarithmic curvatures of the real and complex parts of a real algebraic hypersurface. 
\item The fact that this inequality is sharp in the non-singular case.

\end{enumerate}
\end{abstract}

\maketitle


\section{Introduction}\label{Sec:Intro}

For an affine real smooth algebraic hypersurface $X \subset \CC^{n+1}$ with real part $\RR X \subset \RR^{n+1}$ there is an universal inequality between the total curvatures of the real part $\RR X$ and the one of $\CC X$ (\cite{Ris03}), similar to the ``Smith's Inequality'' between the sums of mod.2 Betti numbers. In this paper we prove an analogous result in the tropical setting; it turns out that in the non-singular tropical case, this inequality becomes an equality; this fact is true in the algebraic world only up to
any positive $\epsilon$ and for special maximal varieties.
 For plane algebraic curves the only cases when the equality holds are the line, the ellipse and the parabola.\\
Let us describe briefly the main results of the paper.\\
If $V$ is a tropical hypersurface defined by a polynomial with coefficients in the field  of real Puiseux series, it has a real part $\RR V$ (see Subsection~\ref{Subsec:Real-Tropical-Hypersurfaces}). Using the fact that $V$ (resp. $\RR V$) is limit of Amoebas (Resp. Real Amoebas), we define the total curvature of $V$ (resp. $\RR V$) by using the total curvature of Amoebas and passing to the limit.\\
For the real part $\RR V$, which is a polyhedral manifold, we also define its total curvature geometrically (as a polyhedral hypersurface), and call it the ``polyhedral total curvature''.\\
Non-singular tropical hypersurfaces are those whose dual subdivision is primitive; they are the tropical counterpart of primitive T-hypersurfaces of Viro's patchworking (see~\cite{Vir83} and \cite{Vir84}).
The main results we prove about these notions are the following: 
\begin{enumerate}
\item The fact that the notions of total curvature and polyhedral total curvature coincide for real non-singular tropical hypersurfaces (Section~\ref{Sec:Polyhedral-Curvature}). 
\item A universal inequality  between the total curvatures of $V$ and $\RR V$ (Subsection~\ref{Subsec:Real-Tropical-Curvature}) based on a similar inequality between the logarithmic Gaussian curvatures of the complex and real parts of a real algebraic hypersurface (Section~\ref{Sec:Amoebas-Curvature}).
\item In the non-singular case, this inequality turns out to be an equality (Subsection~\ref{Subsec:Real-Tropical-Curvature}).
\item A "Gauss-Bonnet's style" formula for the total curvature of a non-singular (complex) tropical hypersurface (Subsection \ref{Subsec:Gauss-Bonnet}).
\end{enumerate}
We would like to thank the referee for their useful comments.
 
The structure of the paper is the following: 

\smallskip

Section 2 is a preliminary one: it introduces notation and basic properties of tropical and real tropical hypersurfaces.

\smallskip

Section 3 treats the case of Amoebas, using the ``Logarithmic Curvature'' to define their curvature. 

\smallskip

Section 4 describes how we can define the tropical total curvature passing to the limit from the one of Amoebas,  both in the real and complex cases, and contains the main results of the paper.

\smallskip

Section 5 is devoted to defining directly the ``Polyhedral Curvature'' for a 
 real non-singular tropical hypersurface
 and to proving that this technique gives the same notion of total curvature than the previous one in the tropical non-singular case.

\smallskip

Section 6 gives some complements and applications, in particular a Gauss-Bonnet's style formula for non-singular tropical hypersurfaces, comparing the Euler characteristic of a generic complex hypersurface in $(\CC^*)^{n+1}$ and the total curvature of its tropicalisation.


\section{Preliminary}\label{Sec:Prelim}


\subsection{Total curvature }\label{Subsec:Total-Curvature}

In all the paper, $\RR^{n+1}$ is considered with its canonical orientation, 
  $\sigma_n$ will be the volume of the unit sphere $S^n \subset \RR^{n+1}$, and we will set $\displaystyle a_n = \pi \frac{\sigma_{2n}}{\sigma_{2n +1}}$.\\
We have then:
\[
\begin{gathered}
\sigma_{2n} = \frac{2 \times (2 \pi)^n}{1.3\dots (2n-1)}\\
\sigma_{2n+1} = \frac{(2\pi)^{n+1}}{2.4 \dots 2n}\\
a_n = \frac{2.4 \dots 2n}{1.3 \dots (2n-1)}.
\end{gathered}
\]

Let $\RR W \subset \RR^{n+1}$ be a smooth oriented hypersurface, $g: \RR W \rightarrow \RR\PP^n$ the Gauss map $x \mapsto n_x$, where $n_x$
 is a non-zero normal vector to $\RR W$ at $x$.
 We define the curvature function $x \mapsto k(x)$ on $\RR W$ as the jacobian of $g$.
The curvature of a measurable  set $U \subset \RR W$ is the integral  $\int_{U} \vert k(x) \vert dv$
of the curvature function on $U$.
The {\bf total curvature} of $ \RR W$ is then by definition: 

$$ \int_{\RR  W} \vert k(x) \vert dv $$
where $dv$ is the canonical Euclidean volume form on $\RR  W \subset \RR^{n+1}$.
It clearly satisfies the following equality~:
\begin{equation} \label{totcurv}
\int_{\RR  W} \vert k(x) \vert dv =  \int_{\RR\PP^n} \# g^{-1}(\beta) ds
\end{equation}
where $ds$ is the canonical volume form on $\RR\PP^n$ and assuming $g^{-1}(\beta)$ is almost everywhere finite.

One can think of the total curvature as the volume of Im$(g)$ taken ``with multiplicities'', the multiplicity of a point $x$ being the cardinality of the fiber $g^{-1}(x)$.\\

If now $W \subset \CC^{n+1}$ is a smooth analytic complex hypersurface, one may define its curvature  as the ``Lipschitz- Killing'' curvature $K(x) dw$, where $dw$ is the canonical volume form on $W \subset \CC^{n+1}$ and $K : W \rightarrow \RR$ the curvature function.\\

Another way, due to Milnor, to define the function $K$ is the following.
Let $\gamma_{\CC} : W \rightarrow \CC P^n$  be  the ``complex Gauss map'' $x \mapsto [N_x W]$, where $N_x W$ is the complex normal vector to $W$ at $x$. One then has:

\begin{equation} \label{compcurv}
(-1)^n K dw = a_n \gamma_{\CC}^{*} (dp)
\end{equation}
where $dp$ is the volume form on $\CC P^n$ (see \cite{Lan79}); note that $(-1)^n K(x)$ is a positive function on $W$.\\
The following inequality is  proved in \cite{Ris03} in the algebraic case:
 \begin{equation} \label{ineqcurv}
  \frac{\sigma_{2 n}}{\sigma_n} \int_{\RR W} \vert k \vert dv \leq \int_W \vert K \vert dw
 \end{equation}
 This paper is devoted to defining similar notions and proving similar results in the tropical case. 
In particular we prove the same type of inequality for the logarithmic Gaussian curvatures which are the natural curvatures in the tropical setting and study its sharpness.


\subsection{Tropical hypersurfaces }\label{Subsec:Tropical-Hypersurfaces}

In order to fix notations and definitions we use in this text, we recall briefly basic notions in tropical geometry. We use the following notation:
the scalar product is written $z\cdot v$; for $X=(X_1,\dots,X_{n+1})$
 and $\alpha=(\alpha_1,\dotsc,\alpha_{n+1})$ we write
 $X^\alpha:=\prod X_i^{\alpha_i}$; the set of vertices of a polytope $\triangle$ is denoted $Vert(\triangle)$. We tend to identify a point and its coordinate vector when it makes the notation less cumbersome. 

We consider the tropical semi-field $\TT=(\RR\cup\{-\infty\},"+","\cdot")$, where the tropical operations are defined by $u"+"v=max\{u,v\}$ and $u"\cdot"v=u+v$. A tropical polynomial $f\in\TT[X_1,\dots,X_{n+1}]$ is of the form 
$$f(X)="\sum_{\alpha\in\mathcal{E}(f)}u_\alpha X^\alpha"=max_{\alpha\in\mathcal{E}(f)}\{ u_\alpha+X\cdot \alpha\}$$
 where $\mathcal{E}(f)$ denotes the set of exponents of $f$. The Newton polytope of $f$ will be denoted by $\triangle_f$.

Given a tropical polynomial $f$ and a point $\omega\in\RR^{n+1}$, the $\omega$-initial part of $f$ is $$In_\omega f(X):="\sum_{\alpha\in\mathcal{E}(f)\mid f(\omega)="u_{\alpha}\omega^\alpha"}u_{\alpha}X^\alpha".$$

A tropical polynomial $f$ determines naturally a polyhedral subdivision of $\RR^{n+1}$ whose cells are formed of the sets of points defining the same initial part of $f$.

Dually, the set of cells
$$\Gamma_f:=(\triangle_{In_\omega f})_{\omega\in\RR^{n+1}}$$ 
 realises a regular subdivision of $\triangle_f$. It is dual to the subdivision of $\RR^{n+1}$ determined by $f$ in the following usual sense. If $c$ is a cell of the subdivision of $\RR^{n+1}$ determined by $f$ there exists a unique dual cell $\check{c}$ of $\Gamma_f$. It satisfies $dim(c)+dim(\check{c})=n+1$ and $c$ and $\check{c}$ generate orthogonal affine subspaces.

We say that a tropical polynomial $f$ is \textbf{generic} if $\Gamma_f$ is simplicial, \textbf{non-singular} if all the maximal dimensional cells of $\Gamma_f$ are primitives simplices, and \textbf{primitive} if $\triangle_f$ is a primitive simplex of dimension $n+1$.

The \textbf{corner locus} of a tropical polynomial $f\in\TT[X]$ is the set of points $\omega\in\RR^{n+1}$ where the value of $f$ is attained by at least two of its monomials. Or equivalently, the set of points contained in cells of dimensions at most $n$. 
 In this text, such a set will be called a \textbf{tropical hypersurface}.

\begin{defi}\label{Def:Tropical-Hypersurface}Let $f\in\TT[X]$. The set 
 $$V(f):=\{\omega\in\RR^{n+1}\mid In_\omega f\text{ is not a monomial}\}$$
is by definition the tropical hypersurface defined by $f$.
\end{defi}

We denote by $Vert(V(f))$ the $0$-dimensional cells of $V(f)$ and more generally $Vert(Z)$ will always denote the $0$-dimensional cells of the natural subdivision of a piecewise linear variety $Z$.

\subsection{Real convergent Puiseux series and tropicalisation }\label{Subsec:Tropicalisation}

A formal series 
$$\xi(t)=\sum_{r\in R}\beta_rt^r$$
 is a \textbf{locally convergent generalised Puiseux series} if $R\subset\RR$ is a well-ordered set, $\beta_r\in\CC$, and the series is convergent for $t>0$ small enough. 
 Denote by $\KK$ the set of all locally convergent generalised Puiseux series. It is an algebraically closed field of characteristic $0$. A series $\xi\in\KK$ is said to be \textbf{real} if all its coefficients are real numbers. We denote by $\KK_{\RR}$ the subfield of $\KK$ composed by the real series.

Since the coefficients of a polynomial $F\in\KK[x_1,\dots,x_{n+1}]$ are locally convergent near $0$, any polynomial over $\KK$ (resp. $\KK_\RR$) can be thought as a one parametric family of complex (resp. real) polynomials. For any $t, \; 0<t \ll 1$, the complex (resp. real) polynomial $F_t$ is the 
 polynomial resulting of the evaluation of the coefficients at $t$.

By hypothesis, the set of exponents of an element of $\KK^*$ has a first element. The map that sends 
 each element of $\KK^*$ to the first element of its set of exponents and $0$ to $\infty$ is a non-archimedean valuation.
 We denote by $val$ the opposite of such a map. In other words, $$val: \KK\rightarrow \RR\cup\{-\infty\}$$ maps $\sum_{r\in\RR}\beta_r t^r\neq 0$ 
 to $-min\{r\mid \beta_r\neq 0\}$ and $0$ to $-\infty$. The map $val$ extends naturally to the map $Val:\KK^{n+1}\rightarrow (\RR\cup\{-\infty\})^{n+1}$ by 
 applying $val$ coordinate-wise. The image of a variety $V\subset \KK^{n+1}$ under the map $Val$ is called the \textbf{non-archimedean amoeba} of $V$. 

Given a polynomial $F\in\KK[x_1,\dots,x_{n+1}]$ one can associate a tropical polynomial. This map is called tropicalisation and acts as follows: If $F(x)=\sum_{\alpha\in\mathcal{E}(F)}c_\alpha x^\alpha$, the \textbf{tropicalisation} of $F$ is the tropical polynomial 
$$Trop(F)(X):="\sum_{\alpha\in\mathcal{E}(F)}val(c_\alpha)X^\alpha".$$
 Kapranov's theorem establishes that the non-archimedean amoeba of $F$ and the tropical hypersurface $V(Trop(F))$ 
 coincide.

Given a tropical hypersurface $V \subset \RR^{n+1}$ and a polynomial $F\in\KK[x_1,\dots,x_{n+1}]$, we will say that $F$ {\it realises} $V$ if $V(Trop(F))=V$.


\subsection{Real tropical hypersurfaces }\label{Subsec:Real-Tropical-Hypersurfaces}

Real tropical hypersurfaces are very closely related to Viro Patchworking (See~\cite{Vir83} and \cite{Vir84}). A description can be found in \cite{Ber2} and one can look at \cite{Mikh04a} pp.~25 and 37, \cite{Vir01}, and
the appendix of \cite{Mikh00} in the case of amoebas for further details.

We recall some definitions here for the convenience of the reader.

Let $F \in \KK_\RR[x_1, \dotsc, x_{n+1}]$ be a {\bf real} polynomial defined over the field of real Puiseux series. 

Let $wal: (\KK^*) \to \RR\times S^1$ be the map sending a real Puiseux series  $\xi(t)=\sum_{r\in R}\beta_rt^r$ to $(val(\xi(t)),arg(\beta_{-val(\xi(t))}))$ and $Wal: (\KK^*)^{n+1} \to \RR^{n+1}\times{\left(S^1\right)}^{n+1}$ be the map defined by $wal$ coordinate-wise.
 The map $wal$ restricts to $wal_\RR: (\KK_\RR^*) \to \RR\times\ZZ_2$ which sends  $\xi(t)=\sum_{r\in R}\beta_rt^r$ to $(val(\xi(t)),sign(\beta_{-val(\xi(t))}))$ and we denote $Wal_\RR: (\KK_\RR^*)^{n+1} \to \RR^{n+1}\times\ZZ_2^{n+1}$ the corresponding restriction of $Wal$.
For any $z\in\ZZ_2^{n+1}$ we will call {\bf orthant} of the torus over the field of Puiseux series and denote by $Q_z^{\KK_\RR}$ the preimage of $\RR^{n+1} \times \{z\}$ under $Wal_\RR$. As in the case of $(\RR^*)^{n+1}$ an orthant is thus a choice of  sign for each coordinate.
The map $Wal$ allows to consider the collection of images under $Val_\RR$ of each orthant  $Q_z^{\KK_\RR}$.

\begin{defi}\label{Def:Real-Tropical-Hypersurface-1}
A {\bf real tropical hypersurface} $V^\RR(Trop(F))$ is the data of $Wal(V(F))$ for a polynomial $F \in \KK_\RR[x_1, \dotsc, x_{n+1}]$.
The {\bf real part} of the real tropical hypersurface is $Wal_\RR(V(F)\cap (\KK_\RR^*)^{n+1})$.
\end{defi}

Let us now compare this definition to the following patchworking procedure.

Let $f$ be a generic tropical polynomial, and  $\vartheta:\mathcal{E}(f)\rightarrow \{1,-1\}$ be a distribution of signs. Let $(e_i)_{i=1..n+1}$ be the canonical basis of $\ZZ^{n+1}\subset\RR^{n+1}$.
For $z=\sum_{i=1}^{n+1}z_i e_i\in\ZZ^{n+1}$,
 let $s_z:\RR^{n+1}\to\RR^{n+1}$ be the symmetry mapping 
 $x=\sum_{i=1}^{n+1}x_i e_i$ to 
 $s_z(x)=\sum_{i=1}^{n+1}(-1)^{z_i}x_i e_i$. 
The map $s_z$ only depends on the reduction modulo $2$ of the coordinates of $z$, so we will indifferently use the notation $s_z$ for $z \in \ZZ^{n+1}$ or $z \in \ZZ_2^{n+1}$.
 
Define the symmetrised distribution of sign $S_z(\vartheta):\mathcal{E}(f)\rightarrow \{1,-1\}$ by $S_z(\vartheta)(v) = (-1)^{z\cdot v}\vartheta(v)$. The maps $S_z$ are involutions on the set $\{\vartheta\}$ of sign distributions on $\mathcal{E}(f)$. They define an action of $(\ZZ_2)^{n+1}$ on sign distributions. We will consider this action via the maps $S_z$ below in Subsection~\ref{Sec:Polyhedral-Curvature} (for example in Proposition~\ref{Prop:Transitivity}).

\begin{defi}\label{Def:Real-Tropical-Hypersurface} Let $f$ be a generic tropical polynomial, and  $\vartheta:\mathcal{E}(f)\rightarrow \{1,-1\}$ a distribution of signs.

The {\bf patchworked real tropical hypersurface} $V^\RR_\vartheta(f)$ is 
the  data of $V(f)$ and $\vartheta$. The {\bf real part } $\RR V^\RR_\vartheta(f)$ of $V^\RR_\vartheta(f)$ is a subset of $\RR^{n+1}\times \ZZ_2^{n+1}$ consisting of relevant symmetric copies of cells of $V(f)$.
Namely for each given $z \in \ZZ_2^{n+1}$ and $c$ a cell of $V(f)$, $s_z(c) \subset \RR V^\RR_\vartheta(f)$ if and only if 
 $$S_z(\vartheta)(Vert(\check{c})) = \ZZ_2.$$
\end{defi}

Here and below we identify the set of elements of $\ZZ_2$ with $\{1,-1\}$ or $\{+,-\}$ depending on which is more convenient. 
With this identification, the above equality   $S_z(\vartheta)(Vert(\check{c})) = \ZZ_2$ just means that not all vertices of $\check{c}$ carry the same sign.

\begin{rem}
The set $\RR V^\RR_\vartheta(f)\cap \left(\RR^{n+1}\times \{z\}\right)$ is either empty or a polyhedral hypersurface.
\end{rem}

\begin{rem}
One can identify the real tropical hypersurfaces $V^\RR_\vartheta(f)$ and $V^\RR_{-\vartheta}(f)$ 
 one being obtained from the other by simultaneously reversing all signs ({\it i.e.}, multiplying $\vartheta(v)$ by $-1$ for all $v$).
The real parts $\RR V^\RR_\vartheta(f)$  and $\RR V^\RR_{-\vartheta}(f)$ are the same.
\end{rem}

If $F \in \KK_\RR[x_1,\dots,x_{n+1}]$ is a polynomial with real series coefficients, we associate to each monomial the sign of the first term of its coefficient. In particular this defines a natural sign distribution $\vartheta_F:\mathcal{E}(Trop(F))\rightarrow \{1,-1\}$ at the vertices of the  subdivision $\Gamma_{Trop(F)}$. The proposition below is a direct consequence of Viro's Patchworking.

\begin{prop}\label{Prop:Patchworking}
Let $F\in \KK_\RR[x_1,\dots,x_{n+1}]$ be a polynomial
 such that  $Trop(F)(X)="\sum_{\alpha\in\mathcal{E}}u_\alpha X^\alpha"$ is generic.
 For $\omega \in \RR^{n+1}$ let $\alpha$ be in $\Gamma_{In_\omega Trop(F)}$.
 If for every such pair $(\omega,\alpha)$ the identity $Trop(F)(\omega)="u_{\alpha}\omega^\alpha"$ holds only if $\alpha$ is a vertex of $\Gamma_{Trop(F)}$ then, $Wal_\RR(V(F)\cap (\KK_\RR^*)^{n+1})=\RR V^\RR_{\vartheta_F}(Trop(F))$ \mbox{i.e.}, the real part of the real tropical hypersurface coincides with the real part of the patchworked tropical hypersurface.
\end{prop}

\begin{rem}
In order to recover a hypersurface in $\left(\RR^*\right)^{n+1}$ 
 from $\RR V^\RR_\vartheta(f)$
 one  uses the map  $\mathfrak{exp}:\RR^{n+1}\times\ZZ_2^{n+1}\to\left(\RR^*\right)^{n+1}$
  defined by $\mathfrak{exp}((x,z)):=s_z(exp (x))$ where the exponential is applied component-wise.
\end{rem}

\begin{exa}\label{Exa:Real-Conic}
Let $f_c$ be a second degree tropical polynomial $"a_0 + a_1 X +a_2 Y + a_3 X^2  + a_4  XY + a_5 Y^2"$ such that $V(f_c)$ is the tropical conic represented on Figure~\ref{Fig:Newton}~a) and $\vartheta$ be the signed distribution shown on Figure~\ref{Fig:Newton}~b). On Figure~\ref{Fig:Real-Conic}, we have drawn the real part $\RR V^\RR_\vartheta(f_c)$ of the real tropical conic $V^\RR_\vartheta(f_c)$ in each of the four quadrants $\RR^2\times\{z\}$ corresponding to the four elements of $\ZZ_2^2$. The axes are the dashed blue lines. The dotted-dashed black segments are the parts of the symmetric copies $s_z(V(f_c))$ of $V(f_c)$ which do not belong to $\RR V^\RR_\vartheta(f_c)$. The real part $\RR V^\RR_\vartheta(f_c)$ is depicted in plain thick red. In each region of $\RR^2 \setminus s_z(V(f_c))$ we indicated the sign of each vertex in the sign distribution $S_z(\vartheta)$.
On Figure~\ref{Fig:Exp} we represented the image of $\RR V^\RR_\vartheta(f_c)$ under the map $\mathfrak{exp}$ (still in plain thick red).
\end{exa}

\begin{figure}
\begin{tabular}{cc}
\includegraphics[width=0.3\textwidth] {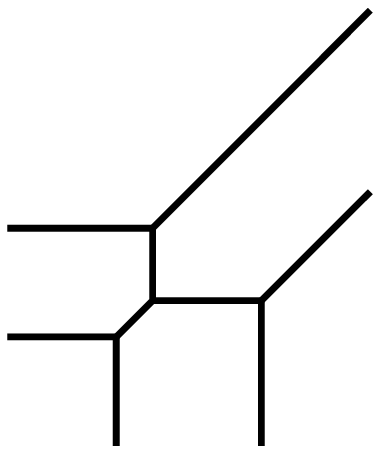}&
\includegraphics[width=0.4\textwidth]{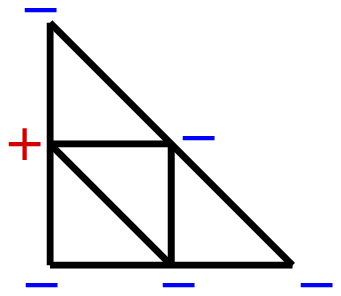}\\
a) A tropical conic & b) Its Newton polygon with a distribution of signs
\end{tabular}
\caption{A tropical conic $V(f_c)$ and the corresponding triangulation of its Newton polygon  with a sign distribution $\vartheta$ at its vertices.}
\label{Fig:Newton}
\end{figure}

\begin{figure}
\begin{tabular}{cc}
\includegraphics[width=0.4\textwidth]{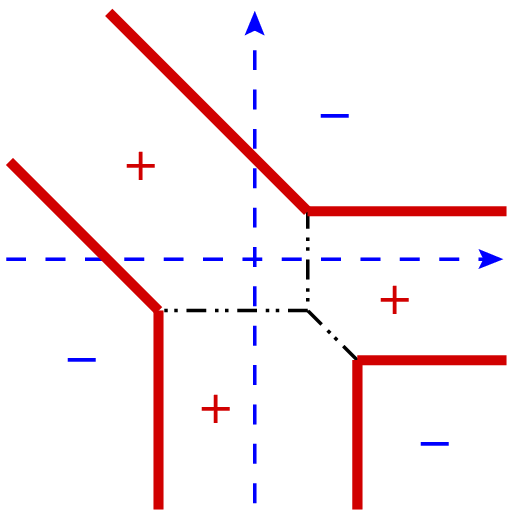}&
\includegraphics[width=0.4\textwidth]{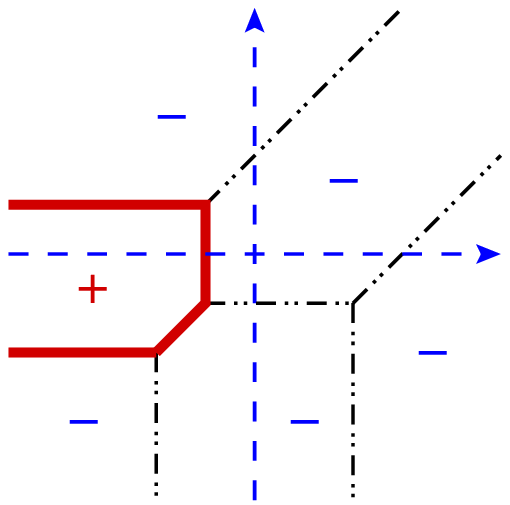}\\
a) In  $\RR^2\times\{(1,0)\}$ & b) In the first quadrant $\RR^2\times\{(0,0)\}$\\ 
\includegraphics[width=0.4\textwidth]{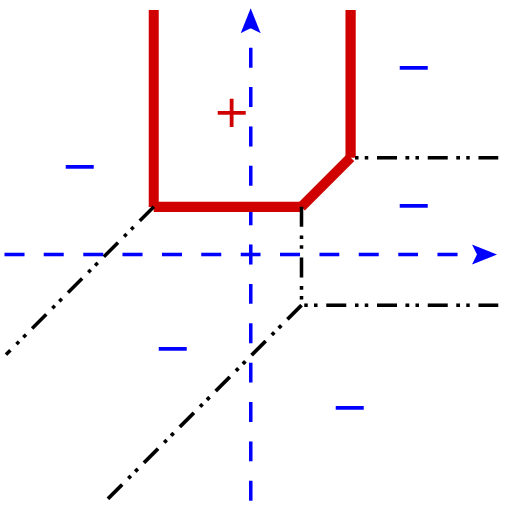}&
\includegraphics[width=0.4\textwidth]{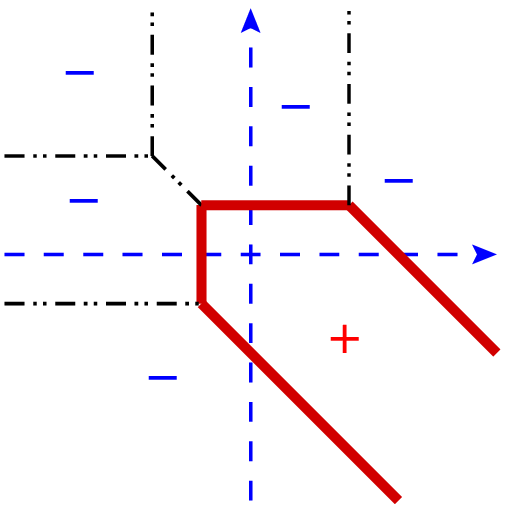}\\
c) In  $\RR^2\times\{(1,1)\}$ & d) In $\RR^2\times\{(0,1)\}$ 
\end{tabular}
\caption{The real part
 $\RR V^\RR_\vartheta(f_c)$ of the real tropical conic $V^\RR_\vartheta(f_c)$ in the four quadrants (after the relevant reflections). The signs of the corresponding vertices of the dual triangulation are indicated in each region of the plane.}\label{Fig:Real-Conic}
\end{figure}

\begin{figure}
\includegraphics[width=\textwidth] {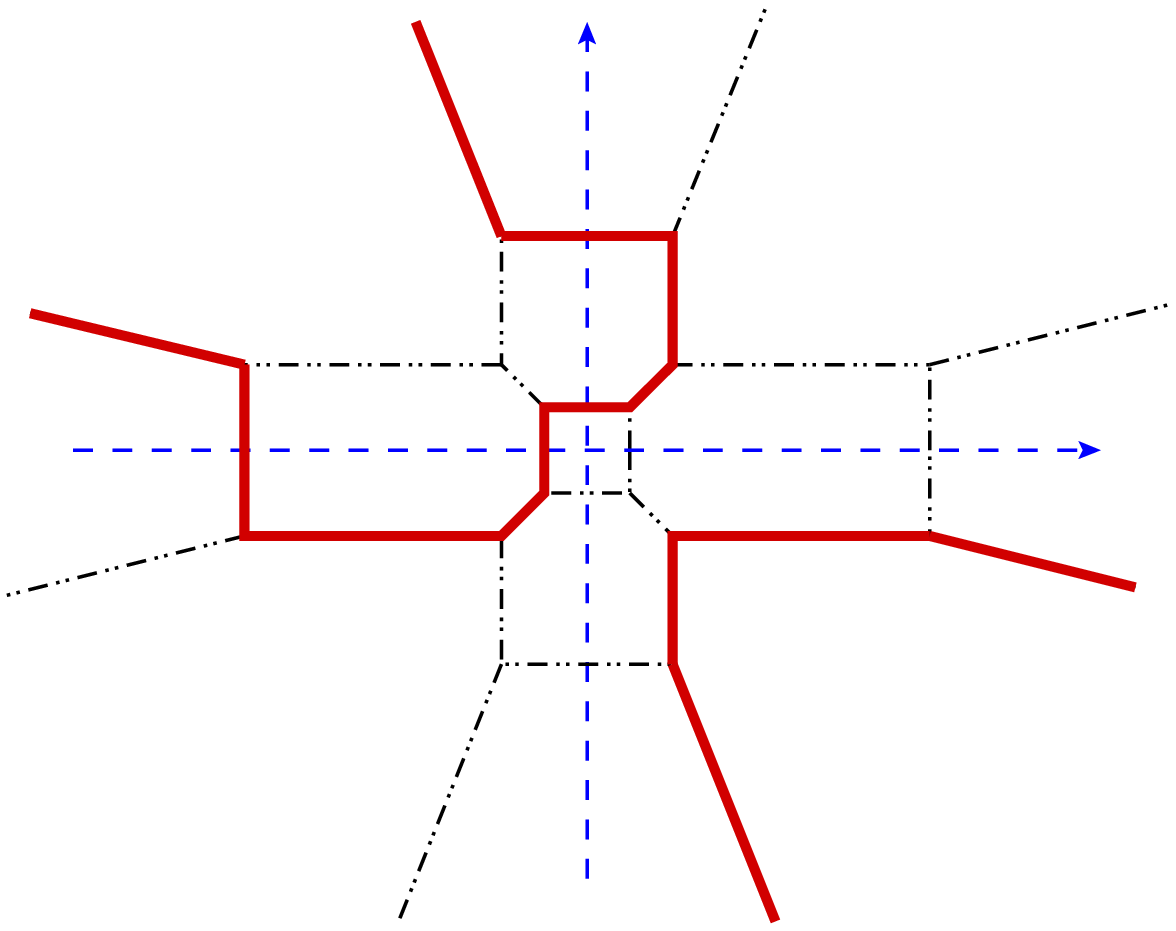}
\caption{The real tropical conic $V^\RR_\vartheta(f_c)$ after applying the map $\mathfrak{exp}$. }\label{Fig:Exp}
\end{figure}

When no confusion is possible we will abuse notation and terminology and write $V^\RR_\vartheta(f)$ instead of $\RR V^\RR_\vartheta(f)$ and "real tropical hypersurface" instead of "real part of the real tropical hypersurface". For a vertex $\textbf{v}$ of $V^\RR_\vartheta(f)$,  when $(\textbf{v},z)$ is a vertex  of $\RR V^\RR_\vartheta(f)$, it will be implicitly denoted by $s_z(\textbf{v})$. Notice that if all vertices of $\check{\textbf{v}}$ have the same sign, $(\textbf{v},z)$ is not in  $\RR V^\RR_\vartheta(f)$.



\section{Total curvature of real and complex amoebas}\label{Sec:Amoebas-Curvature}

Let  $F\in \KK_{\RR}[x_1,\dots,x_{n+1}]$ be a polynomial defining a non-singular hypersurface in ${(\KK^*)}^{n+1}$. For $t\in\RR^*$, $\vert t \vert \ll 1$, $F_t\in\RR[x_1,\dots,x_{n+1}]$ defines a non-singular hypersurface  $V(F_t)\subset(\CC^*)^{n+1}$. We call it a real algebraic variety because $F_t$ is defined over $\RR$ and denote by $\RR V(F_t)\subset (\RR^*)^{n+1}$ the real part of $V(F_t)$.\\

We will set $\triangle_F$ for the Newton polygon of $F$ (or $F_t$ for $0 < t \ll  1$). 
 We denote by $\Log_t$ the map from $(\CC^*)^{n+1}$ to $\RR^{n+1}$ that sends $(x_1,\dotsc, x_{n+1})$ to $(\log_t(|x_1|),\dotsc, \log_t(|x_{n+1}|))$ where $\log_t$ is the base $t$ logarithm.

\begin{defi}\label{Def:Amoebas}
Set $x = (x_1, \dots, x_{n+1})$.
\par 1) For $F \in \KK[x], 0<t \ll  1$, $F_t \in \CC[x]$, the amoeba of $V(F_t)$ is 
\[ \mathcal{A}(F_t) = \Log_t(V(F_t)) \subset {\RR}^{n+1}. \]
\par 2) For $F \in{\KK}_{\RR}[x],  0<t \ll  1$, $F_t \in \RR[x]$, The real amoeba of $\RR V(F_t)$ is
\[  \mathcal{A}^\RR(F_t) = \Log_t(\RR V(F_t)). \]

\end{defi}

\begin{rem} In each orthant $Q \subset (\RR^*)^{n+1}$, the map $\Log_t \vert_{Q}$ is a diffeomorphism onto $\RR^{n+1}$; then we may define the `` Gauss map `` $g_t : \mathcal{A}^\RR(F_t) \rightarrow \RR \PP^n$'' by taking  the Gauss map for the image of each orthant (then for some points of the Amoeba, the ``map'' $g_t$  may be multivalued).
\end{rem}

We then have the following diagram:
\begin{eqnarray}
\label{Diag:Log-Gauss}
\xymatrix{ \RR V(F_t) \ar[rd]^{\gamma_t^\RR} \ar[d]^{\Log_t} \\
\mathcal{A}^\RR(F_t) \ar[r] _{g_t} & \RR \PP^n}
\end{eqnarray}

where $g_t$ is the Gauss map and $\gamma_t^\RR=g_t\circ \Log_t$ the logarithmic Gauss map, defined as:
$$\gamma_t^\RR (x)=[x_i\frac{\partial F_t}{\partial x_i}].$$

\begin{rem}\label{Def:Amoebas-Real-Curvature} According to diagram (\ref{Diag:Log-Gauss}), the  total curvature of the amoeba $\mathcal{A}^\RR(F_t)$ is 
$$\int_{\mathcal{A}^\RR(F_t)} \vert k \vert dv = \int_{\RR\PP^n} \# {\left(\gamma_t^\RR\right)^{-1}}(\beta) ds $$ 
\end{rem}

 In the complex case, contrary to the real case, the Amoeba $\mathcal{A}(F_t)$ is not in general an immersed manifold, therefore there is no natural definition of a Gauss map $\mathcal{A}(F_t) \rightarrow \CC \PP^n$.
 However, the Logarithmic Gauss map $\gamma_t$: $ V(F_t) \rightarrow  \CC \PP^n$ is meaningfull, and we have the following diagram:

\[
\xymatrix{ V(F_t) \ar[rd]^{\gamma_t} \ar[d]^{\Log_t} \\
\mathcal{A}(F_t)  & \CC \PP^n}
\]
It is then natural to give the following definition (see (\ref{compcurv})):

\begin{defi}\label{Def:Amoebas-Complex-Curvature}The total curvature of the amoeba $\mathcal{A}(F_t)$ is defined by 
$$\int_{\mathcal{A}(F_t)}  K  : = (-1)^n a_n \int_{\CC \PP^n} \#  \gamma_t^{-1}(\beta) dp.$$ 
\end{defi}

\begin{rem}\label{Rem:Multiplicative-Translation-Invariance}
We consider subvarieties of the torus $(\CC^*)^n$ therefore for any $\omega\in\ZZ^{n+1}$, $F_t$ and $x^\omega F_t$ define the same variety.
Moreover, the logarithmic Gauss map 
$$\fonction{\gamma}{V(F_t)}{\CC\PP^n}{x}{[x_i\frac{\partial F_t}{\partial x_i}]}$$
is also invariant when multiplying $F_t$ by $x^\omega$ (since the tangent space to $V(F_t)$ at a given point is unaffected by this action).
\end{rem}

These definitions lead us to consider the systems 
\[ (G'_\beta)  \begin{cases}  F_t(x) = 0 \\ \displaystyle [x_1\frac{\partial F_t}{\partial x_1}: \dotso : x_{n+1}\frac{\partial F_t}{\partial x_{n+1}}] = [\beta_1:\dotso :\beta_{n+1}] \end{cases} \]
where  $[\beta]=[\beta_1:\beta_2:\dotso :\beta_{n+1}]$ is an element
 of $\CC \PP^{n+1}$.
Since we only consider solutions in the torus, the system $(G'_\beta)$
is equivalent to 

\[ (G_\beta)  \begin{cases}  F_t= 0 \\ \displaystyle x_1\frac{\partial F_t}{\partial x_1}=\beta_1y 
\\ \displaystyle \vdots \\ \displaystyle x_{n+1}\frac{\partial F_t}{\partial x_{n+1}}=\beta_{n+1}y \end{cases} \]
where we introduce a new variable $y\in\CC^*$. 

\begin{defi}
We say that a polynomial $F_t$ is B-generic if, for a generic $\beta$, the system $(G_{\beta})$ satisfies the genericity conditions of \cite{Bern75}
{\it i.e.}, the restriction of the system to any proper face of the convex hull of its set of exponent has no solution in the torus. 
\end{defi}

\begin{rem}\label{Rem:Multiplicative-Translation-Invariance-2}
Let $(G_{\beta}'')$ be the system
\[ \begin{cases} x^\omega F_t = 0 \\  \displaystyle [x_1\frac{\partial x^\omega F_t}{\partial x_1}: x_2\frac{\partial x^\omega F_t}{\partial x_2}: \dotso : x_{n+1}\frac{\partial x^\omega F_t}{\partial x_{n+1}}] = [\beta_1:\beta_2:\dotso :\beta_{n+1}] \end{cases}.\]

If the Newton polytopes of $F_t$ and $x^\omega F_t$ are included in $(\RR_+^*)^{n+1}$ we find algebraically the statement of Remark~\ref{Rem:Multiplicative-Translation-Invariance}. Indeed, $(G_{\beta}'')$ is then clearly equivalent to:
\[ \begin{cases} F_t = 0 \\  \displaystyle [x_1\frac{\partial F_t}{\partial x_1}: x_2\frac{\partial F_t}{\partial x_2}: \dotso : x_{n+1}\frac{\partial  F_t}{\partial x_{n+1}}] = [\frac{\beta_1}{x^\omega}:\frac{\beta_2}{x^\omega}:\dotso :\frac{\beta_{n+1}}{x^\omega}] \end{cases} \]
since we consider solutions in $(\CC^*)^{n+1}$.
Thus $(G_{\beta}'')$ is equivalent to $(G_{\beta}')$ and for every $[\beta]\in\CC\PP^{n+1}$, both corresponding logarithmic Gauss maps have the same fiber.
\end{rem}

Let $F_t^\delta$ denote the truncation of $F_t$ to a face $\delta$ of the Newton polytope $\triangle_{F_t}$ of $F_t$.

\begin{defi}[Viro]
Let $F\in\CC[x_1,\dotsc,x_{n+1}]$ be a polynomial, $\triangle_{F}$ its Newton polytope and $X_{\triangle_{F}}$ the toric variety associated to
$\triangle_{F}$. We say that $F$ is  {\bf completely non-degenerate} if for each (not necessarily proper) face $\delta$ of $\triangle_{F}$ the restriction $F^\delta$ defines a non-singular variety in $(\CC^*)^{n+1}$.
\end{defi}

\begin{rem}\label{Rem:Tprimitive-nondegenerate}
Let $F\in\KK[x_1,\dotsc,x_{n+1}]$ be such that $Trop F$ is non-singular. Then $F_t$ is completely non-degenerate for $0 < t \ll 1$.
\end{rem}

\begin{prop}\label{Prop:Nondegenerate}
A polynomial $F_t\in\CC[x_1,\dotsc,x_{n+1}]$ is B-generic if and only if it is completely non-degenerate.
\end{prop}

\begin{proof} 

The support polytope $\triangle_\beta$ of the system $(G_\beta)$ is the cone with apex $(0,\dotsc,0,1)$ over $\triangle_{F_t}\times \{0\}$.
In fact, $\triangle_\beta$ is the Newton polytope of all polynomial of $(G_\beta)$ but $F_t$.

The restrictions of $(G_\beta)$  to the faces of $\triangle_\beta$ are either of the form 

\[ (G^\delta_\beta)  \begin{cases}  F_t^\delta = 0 \\ \displaystyle x_1\frac{\partial F_t^\delta}{\partial x_1}=\beta_1 y 
\\ \displaystyle \vdots \\ \displaystyle x_{n+1}\frac{\partial F^\delta _t}{\partial x_{n+1}}= \beta_{n+1} y \end{cases} \]

or of the form, 

\[ (G^\delta_0)  \begin{cases}  F_t^\delta = 0 \\ \displaystyle x_1\frac{\partial F_t^\delta}{\partial x_1}=0
\\ \displaystyle \vdots \\ \displaystyle x_{n+1}\frac{\partial F^\delta _t}{\partial x_{n+1}}= 0 \end{cases} \]

for $\delta$ a face of  $\triangle_{F_t}$.

The system $(G^\delta_\beta)$ is the resctriction of $(G_\beta)$ to the cone with apex $(0,\dotsc,0,1)$ over a face $\delta$ of $\triangle_\beta$.

By definition, $F_t$ is B-generic if and only if for all proper faces $\delta$ of $\triangle_{F_t}$ the systems $(G^\delta_\beta)$ have no solution for a generic
$\beta$
and for all (not necessarily proper) faces   $\delta$ of $\triangle_{F_t}$ the systems  $(G^\delta_0)$  have no solution.

Note that $(G^\delta_0)$  has no solution if and only if $V(F^\delta_t)$ is non-singular in $(\CC^*)^{n+1}$.

Assume first that $F_t$ is B-generic. Then, in particular  for all (not necessarily proper) faces   $\delta$ of $\triangle_{F_t}$ the systems  $(G^\delta_0)$  have no solution.
So, for any face $\delta$, $V(F^\delta_t)$ is non-singular which is equivalent to $F_t$ being completely non-degenerate.

Assume now that $F_t$ is completely non-degenerate. We saw above that this implies that for all (not necessarily proper) faces   $\delta$ of $\triangle_{F_t}$ the systems  $(G^\delta_0)$  have no solution. We only need to check that for a generic $\beta$, the  systems $(G^\delta_\beta)$ have no solution for all proper faces $\delta$ of $\triangle_{F_t}$.  

Let us now fix a proper face $\delta$ of $\triangle_{F_t}$ and consider the system $(G^\delta_\beta)$.
Denote by $\overline\gamma_\delta$ the map $$\overline\gamma_\delta:(\mathbb C^*)^{n+1}\rightarrow\mathbb C^{n+1}$$ $$x\mapsto (x_1\frac{\partial F_t^\delta}{\partial x_1}, \dotsc , x_{n+1}\frac{\partial F_t^\delta}{\partial x_{n+1}}).$$
Note that $0 \in \overline\gamma_\delta(V(F_t^\delta))$ if and only if $V(F_t^\delta)$ is singular. 
 Let us write $F_t(x)= \sum_{\alpha\in\mathcal E(F_t)}a_\alpha  x^\alpha$.
Since $dim(\delta)<n+1$, $\delta$ is contained in a hyperplane and there exist $\mu\in\RR^{n+1}$ and $c\in\RR$ such that for any exponent $\alpha$ of $F_t^\delta$, $$\mu\cdot \alpha=c.$$
 
For any $x\in (\mathbb C^*)^{n+1}$, $$\mu\cdot \overline\gamma_\delta(x)=\sum_{\alpha\in\mathcal E(F_t^\delta)}a_\alpha (\mu\cdot\alpha) x^\alpha$$ $$=\sum_{\alpha\in\mathcal E(F_t^\delta)}a_\alpha c x^\alpha=c F_t^\delta(x).$$

In particular, if $x\in V(F_t^\delta)$, $\mu\cdot \overline\gamma_\delta(x)=0$. Then, $\overline\gamma_\delta(V(F_t^\delta))$ is contained in the linear hyperplane $H^\delta$ orthogonal to $\mu$.

Then, for any $\delta$ the system $(G^\delta_\beta)$ has no solution in the torus if the line $(\beta_1y,\dotsc,\beta_{n+1}y)$ is not contained in $H^\delta$. Since there exists a finite number of hyperplanes containing all the proper faces of $\triangle_{F_t}$,
 $$\cup_{\delta\subset\triangle_{F_t}}\overline\gamma_\delta(V(F_t^\delta))$$
is contained in a finite union of linear hyperplanes. Then, for a generic $[\beta]$ none of the systems in $\{(G^\delta_\beta),\delta \mbox{ proper face of } \triangle_{F_t} \}$ has a solution. Thus $F_t$ is B-generic.
\end{proof}

From now on we consider completely non-degenerate polynomials.

\begin{prop} A polynomial $F_t$ is completely non-degenerate if and only if the degree of the map $\gamma_t$ is $$(n+1)! vol(\triangle_{F_t}).$$
\end{prop}
\begin{proof}
It follows directly from  Proposition~\ref{Prop:Nondegenerate} applying Bernstein theorem to the systems $(G_\beta)$. 
\end{proof}
 
 As a direct consequence of Proposition~\ref{Prop:Nondegenerate}, we have the following corollary.

\begin{cor}\label{Rem:complexamoeba} $F_t$ is completely non-degenerate if and only if $$ \int_{\mathcal{A}(F_t)}  K = (-1)^n a_n (n+1)! vol(\triangle_F) vol(\CC \PP^n)$$ which does not depend on $t$.
 \end{cor}

We have then the following inequality, similar to (\ref{ineqcurv}) for the logarithmic curvatures of the real and complex parts of a real algebraic hypersurface: 
\begin{thm}\label{Prop:Inequality} For any completely non-degenerate $F_t\in\RR[x_1,\dotsc,x_{n+1}]$,
\begin{equation}\label{Eq:Inequality}
\frac{\sigma_{2n}}{\sigma_{n}} \int_{\mathcal{A}^\RR(F_t)}\vert k \vert\leq  \int_{\mathcal{A}(F_t)}\vert K \vert.
\end{equation}
 
\end{thm}

\begin{proof} 
We have seen that $ \int_{\mathcal{A}(F_t)}  K = (-1)^n a_n (n+1)! vol(\triangle_F) vol(\CC \PP^n)$.
For $x \in \RR \PP^n$, the cardinality of the real fiber $(\gamma_t^{\RR})^{-1}$ is smaller than the cardinality of the complex one. So we similarly have that $ \int_{\mathcal{A}^{\RR}(F_t)}\vert k \vert \leq deg(\gamma_t) vol(\RR \PP^n)$, with $vol(\CC \PP^n) = \frac{\sigma_{2n +1}}{\sigma_1}$ and $vol(\RR \PP^n) = \frac{\sigma_n}{2}$ from which the inequality follows.
\end{proof}

 \begin{cor} There is equality in the above theorem if and only if the map $\gamma_t$ is totally real, {\it i.e.}, $\gamma_t^{-1}(x) \subset \RR V(F_t)$ for  $x \in \RR \PP^n$.
 \end{cor}
 
\begin{rem}
One can prove (see \cite{PasRis10}) that for a non-singular real curve, the maximality for the above inequality characterises the Harnack curves in the sense of \cite{Mikh00}.
\end{rem}


\section{Complex and real  total curvature of tropical hypersurfaces}\label{Sec:Tropical-Curvature}


\subsection{Complex total curvature}\label{Subsec:Complex-Tropical-Curvature}

Let $f$ be a tropical polynomial, $F \in \KK[x]$ realising $f$ and such that $F_t$ is completely non-degenerate for $0 < t \ll  1$. Since the total curvature of the Amoeba $\mathcal{A}(F_t)$ does not depend on $t$ for $0 < t \ll  1$, we define the total curvature of the tropical variety $V(f)$ by passing to the limit in the trivial way:

\begin{defi}\label{Def:Complex-Tropical-Curvature}Let $f$ be a tropical polynomial. We define the complex total curvature of $V(f)$ as 
\begin{equation}\label{Eq:Complex-Tropical-Curvature}\int _{V(f)} K  : =  \int_{\mathcal{A}(F_t)} K  = (-1)^n a_n (n+1)! vol(\triangle_F) vol(\CC \PP^n).\end{equation} 
\end{defi}

Notice that $(-1)^n K$ is a positive function, then $\vert K \vert = (-1)^n K$, therefore $\int _{V(f)}\vert K \vert=vol(\CC\PP^n)\times (n+1)! vol(\triangle_f)\times a_n.$ 

\begin{cor}\label{Cor:Hyperplane-Complex-Tropical-Curvature} For any primitive tropical hypersurface $H$, 
\begin{equation}\label{Eq:Hyperplane-Complex-Tropical-Curvature}
\int _{H} K =(-1)^n a_n vol(\CC\PP^n).
\end{equation}
\end{cor}

\begin{cor}\label{Cor:Complex-Tropical-Curvature} Let $f$ be a tropical polynomial and let $\textbf{v}_1,\dots, \textbf{v}_r$ be the set of vertices of $V(f)$. Then $$\int _{V(f)} K =\sum_{i=1}^r\int_{V(In_{\textbf{v}_i}f)} K. $$
\end{cor}


\subsection{Total curvature of real tropical hypersurfaces}\label{Subsec:Real-Tropical-Curvature}

\begin{defi}\label{Def:Real-Tropical-Curvature} 

Let $F \in  \KK_{\RR}[x_1, \dots, x_{n+1}] $ be a real polynomial, $f = Trop (F)$ its tropicalization.\\
The real total curvature of $V^\RR(f)$ is defined as 
$$\int _{V^\RR(f)}\vert k \vert:= \limsup_{t \rightarrow 0} \int_{\mathcal{A}^\RR(F_t)}\vert k \vert.$$ 
\end{defi}

\begin{rem}\label{Rem:limit} The proof of Proposition~\ref{totcurvR} below implies that if $V(f)$ is non-singular,
 $\int_{\mathcal{A}^\RR(F_t)}\vert k \vert$ has a limit when $t \rightarrow 0$.
\end{rem}

Recall the following diagram:

\begin{equation} \label{diag1}
\xymatrix{ \RR V(F_t) \ar[rd]^{\gamma_t^\RR} \ar[d]^{\Log_t} \\
\mathcal{A}^\RR(F_t) \ar[r] _{g_t} & \RR \PP^n}
\end{equation}
where $g_t$ is the Gauss map and $\gamma_t^\RR=g_t\circ \Log_t$ the logarithmic Gauss map, defined as:

$$\gamma_t^\RR (x)=[x_i\frac{\partial F_t}{\partial x_i}].$$

It follows immediately from Theorem~\ref{Prop:Inequality} and Definitions~\ref{Def:Complex-Tropical-Curvature} and~\ref{Def:Real-Tropical-Curvature} that real and complex tropical curvatures satisfy an inequality similar to Inequality~(\ref{ineqcurv}).

\begin{thm}\label{Thm:Tropical-Inequality}Let $F \in \KK_{\RR}[x_1, \dots, x_{n+1}] $ be a real polynomial and $f= Trop (F)$ be its tropicalisation. We have  
\begin{equation}\label{Eq:Tropical-Inequality}
\frac{\sigma_{2n}}{\sigma_n}\int_{V^\RR(f)}\vert k\vert\leq\int_{V(f)}\vert K\vert
\end{equation}
\end{thm}
\qed

We now establish one the main results of this article; namely that the Inequality~(\ref{Eq:Tropical-Inequality}) of Theorem~\ref{Thm:Tropical-Inequality} is an equality for real non-singular tropical hypersurfaces (Theorem~\ref{Thm:Tropical-equality}).
Let us first look at the case of a primitive hypersurface.

\begin{prop}\label{Prop:Realcurv} Let $f$ be a primitive tropical polynomial, tropicalisation of  $F \in  \KK_{\RR}[x_1, \dots, x_{n+1}] $. Then,

\begin{equation}\label{Eq:Realcurv}
 \int _{V^\RR(f)}\vert k \vert=vol(\RR\PP^n)=\frac{\sigma_n}{2}.
\end{equation}
\end{prop}
\begin{proof}

We have $\int_{{\mathcal A}^{\RR}(F_t) }\vert k \vert = $vol$(Im(g_t))$
 by definition. But the map $\gamma_t$ is generically of degree one by Berstein's theorem, and for a generic $\beta \in \RR \PP^n$ (namely in the complementary of the set of normal directions to non-compact cells of the tropical variety $V(f)$), we have  $\#(\gamma_t^{\RR})^{-1}(\beta) \leq \#(\gamma_t)^{-1}(\beta) = 1$ 
and $\#(\gamma_t^{\RR})^{-1}(\beta) \equiv \#(\gamma_t)^{-1}(\beta) \bmod 2$; therefore $\#(\gamma_t^{\RR})^{-1}(\beta) = 1$ and we have that vol($Im(g_t)) = $vol$(Im(\gamma_t^{\RR}))=$ vol$(\RR \PP^n) = \frac{\sigma_n}{2}$, and (\ref{Eq:Realcurv}) passing to the limit.
\end{proof}

Let now $F(x)  = \sum a_i(t) x^{\alpha_i} \in \KK_{\RR}[x_1, \dots, x_{n+1}]$ be a  polynomial such that $f = Trop F$ is non-singular, $(\textbf{v}_i)_{(1 \leq i \leq r)}$ the vertices of $V(f)$, $\triangle_F =\cup \triangle_i$  the subdivision of $\triangle_F$ in simplices dual to $V(f)$. Then for each vertex $\textbf{v}_i$, we set $f^{\textbf{v}_i} =Trop F^{\textbf{v}_i}$, with $F^{\textbf{v}_i} = \sum_{\alpha_j \in \triangle_i} a_j x^j$. Notice that $In_{\textbf{v}_i} f = f^{\textbf{v}_i}$.

The main step in the proof of Theorem~\ref{Thm:Tropical-equality} is the following: \\

\begin{prop} \label{totcurvR} Let $F \in \KK_{\RR}[x_1, \dots, x_{n+1}]$ be such that $f = $Trop$ F$ is non-singular. Then
\[ \int _{V^{\RR} (f)}\vert k \vert=\sum_{i=1}^r\int_{V^{\RR}(f^{\textbf{v}_i})}\vert k \vert = r \; vol(\RR \PP^n).\]
\end{prop} 

Before proving the proposition, we need a lemma of "localisation" at a vertex.\\
\begin{lemma}
Let $v \in V(f)$ be a vertex, $v =  (\lambda_1, \dots, \lambda_{n+1}) \in \RR^{n+1}$. Let $\beta_0 \in  \RR \PP^n $ be a generic element (see the proof of Proposition  \ref{Prop:Realcurv}) and $\eta > 0$ be given. Then there exist $\epsilon  >0$ and $t_0 > 0$ such that for any $\beta$ such that $d(\beta,\beta_0) < \epsilon$ and $t$ such that $0 < t < t_0$, we have  $ B(v, \eta) \cap g_t^{-1}(\beta) \not = \emptyset$. Here $d$ is induced on $\RR \PP^n$ by the distance on the unit sphere in $\RR^{n+1}$.
\end{lemma}

\begin{proof}
Up to multiplying each $x_i$ by $t^{-\lambda_i}$
 (which has the effect of translating the vertex $v$ to $0$) and multiplying $F_t$ by the relevant power of $t$, one may write:
$$F_t(x)=\sum_{\alpha \in Vert(\check{v})} a_i(0) x^\alpha + t^\nu Q_t(x)$$ 
where  $a_i(0)\in \RR^*$,  $\nu$ is a positive real number and $Q_t\in \KK^\RR[x_1,\dotsc,x_{n+1}]$ is a polynomial whose coefficients have non-negative valuation.
 We set $H(x)=\sum_{\alpha \in Vert(\check{v})} a_i(0) x^\alpha$ so that $F_t= H +t^\nu Q_t$.
For $\beta \in \RR \PP^n$, the points of $g_t^{-1}(\beta)$ are the images under $\Log_t$ of the solutions of the system:
\[ (G'_\beta)  \begin{cases}  F_t = 0 \\ \displaystyle [x_1\frac{\partial F_t}{\partial x_1}: \dotso : x_{n+1}\frac{\partial F_t}{\partial x_{n+1}}] = [\beta] \end{cases}. \]
 For a generic $\beta_0$, we know by the proof of Proposition~\ref{Prop:Realcurv} that the system: 
\[  \begin{cases}  H = 0 \\ \displaystyle [x_1\frac{\partial H}{\partial x_1}: \dotso : x_{n+1}\frac{\partial H}{\partial x_{n+1}}] = [\beta_0] \end{cases} \]
has a non-degenerated solution $x_0$.
Then the system:
\[  (G^H_\beta) \begin{cases}  H = 0 \\ \displaystyle [x_1\frac{\partial H}{\partial x_1}: \dotso : x_{n+1}\frac{\partial H}{\partial x_{n+1}}] = [\beta] \end{cases} \]
has a solution $x_H$ such that $\Log(x_H)$ is in the ball $B(\Log(x_0),1)$ for $d(\beta,\beta_0)$ sufficiently small.
  
The system $(G'_\beta)$ is a one parameter deformation of $(G^H_\beta)$ therefore it has a solution $x_F$ such that $\Log(x_F)\in B(\Log(x_0),2)$ for $t$ small enough.
Then $\Log_t(x_F)$ is in the ball $\displaystyle B(\frac{\Log(x_0)}{\log t},\frac{2}{|\log t|})$ which is included, for $t$ small enough, in  $B(0, \eta)\, =\, B(v, \eta)$.
\end{proof}

Let us now prove Proposition~\ref{totcurvR}.
 
Let $\Omega \subset \RR \PP^n$ be a compact set, $t_0 > 0$ and  $\epsilon > 0$ such that:
\par 1) Vol$(\RR \PP^n \setminus \Omega) < \epsilon$ 
\par 2) For any direction $\beta \in \Omega$ and any vertex $v$ of $V(f)$, the system $(G'_\beta)$  has a non-degenerated  solution.\\
Then, if $r$ is the number of vertices of $V(f)$, 
 we have that for $0 < t < t_0$:
\[ \int_{{\mathcal A}^{\RR}(F_t)} \vert k \vert \geq r vol (\RR \PP^n) - r \epsilon \]
passing to the limit when $t_0 \rightarrow 0$ and $\epsilon \rightarrow 0$ gives the result.

We now deduce easily the main result of the paper.

\begin{thm}\label{Thm:Tropical-equality} Let $F \in \KK_{\RR}[x_1, \dots, x_{n+1}]$ be such that $f = $Trop$ F$ is non-singular.
 Then
\begin{equation}\label{Eq:Tropical-equality}
 \frac{\sigma_{2n}}{\sigma_n} \int_{V^\RR(f)}\vert k\vert = \int_{V(f)}\vert K\vert
\end{equation}
\end{thm}
\begin{proof}  By Proposition~\ref{totcurvR} and Corollary~\ref{Cor:Complex-Tropical-Curvature}, it is enough to prove the theorem in the primitive case. We have then to prove that $\frac{\sigma_{2n}}{\sigma_n} vol(\RR \PP^n) = a_n vol(\CC \PP^n)$ with $vol(\CC \PP^n) = \frac{\sigma_{2n + 1}}{\sigma_1}$, that is immediate.
\end{proof}


\section{Polyhedral total curvature of a real tropical hypersurface.}\label{Sec:Polyhedral-Curvature}

\subsection{Definition and elementary properties}

Our definition of the curvature in the polyhedral case is similar to  Banchoff's in \cite{Ban70} (see also \cite{Ban67} and \cite{Ban83}) but, exactly as in the complex case,  we only consider absolute value of the curvature here.

It amounts to the following.
The {\bf solid angle} of a cone is the portion of the unit sphere centred at the vertex of the cone that it intersects, its {\bf measure} is the volume of this spherical portion. We might abuse terminology and write "solid angle" when we mean its measure.

Let $\textbf{v}$ be a vertex of a polyhedral hypersurface $\mathfrak{H}$ (here our real tropical hypersurface $\RR V^\RR_\vartheta(f)$). For sufficiently small neighbourhoods $U$ of $\textbf{v}$, $U \setminus \mathfrak{H}$ has two connected components. Label one by $+$ and the other by $-$ (for $\RR V^\RR_\vartheta(f)$  these will be the signs of the corresponding vertices of the dual subdivision). For each maximal dimensional cell of $\mathfrak{H}$ containing $\textbf{v}$, choose a 
normal vector oriented from $-$ to $+$. The {\bf curvature cone} $C_\textbf{v}$ at $\textbf{v}$ is the cone generated by these vectors.  

\begin{defi}\label{Def:Local-Polyhedral-Curvature}
The {\bf curvature} $\kappa_{\textbf{v}}$ at $\textbf{v}$ is the measure of the solid angle of the curvature cone $C_\textbf{v}$. 
\end{defi}

\begin{rem}\label{Rem:Curvature-cone}
For $V^\RR_\vartheta(f)$ the normal vectors above can be chosen to correspond to the vectors $n_i$
 supported by the edges of the simplex $\check{\textbf{v}}$  dual to $\textbf{v}$ and oriented from a vertex with $-$ sign  to a vertex with $+$ sign.

Thus the curvature cone $C_\textbf{v}$ is naturally identified to the cone generated by the $n_i$'s and depends only on the simplex $\check{\textbf{v}}$ dual to $\textbf{v}$ and the sign distribution at the vertices of $\check{\textbf{v}}$.
\end{rem}

\begin{rem}\label{Rem:Sign-inversion}
Changing all the vectors to their opposites clearly leaves $\kappa_{\textbf{v}}$ invariant.
\end{rem}

\begin{defi}\label{Def:Polyhedral-Curvature} Let $f$ be a generic tropical polynomial and let $\vartheta$ be a distribution of signs in $\mathcal{E}(f)$. The polyhedral total curvature of $V^\RR_\vartheta(f)$ is $$\int_{V^\RR_\vartheta(f)}\vert k^p\vert:=\sum_{\textbf{v} \in Vert (\RR V^\RR_\vartheta(f))}\kappa_{\textbf{v}}$$
{\it i.e.}, it is the sum of the curvatures at all vertices of the real part of $V^\RR_\vartheta(f)$.
\end{defi}

It follows from the definition that we have the equality below.
\begin{lemma}\label{Prop:By-Def}Let $f$ be a generic tropical polynomial let $\vartheta$ a distribution of signs in $\mathcal{E}(f)$. If $\textbf{v}_1,\dots,\textbf{v}_k$ are the vertices of $V(f)$, then $$\int_{V^\RR_\vartheta(f)}\vert k^p\vert=\sum_{i=1}^{k}\int_{V^\RR_{\vartheta_i}(In_{\textbf{v}_i}f)}\vert k^p\vert,$$ where $\vartheta_i=\vartheta\mid_{ \mathcal{E}(In_{\textbf{v}_i}f)}.$
\end{lemma}

We denote by $C_{s_z(\textbf{v})}$ the curvature cone at the vertex $s_z(\textbf{v})$ of the real tropical hypersurface $\RR V^\RR_\vartheta(f)$. It is the cone generated by vector edges of $\check{v}$ oriented form vertices with minus sign to vertices with plus sign in the sign distribution $S_z(\vartheta)$.
\begin{rem}\label{Rem:alternative-Sum}
One can also define the curvature $\widetilde{\kappa_\textbf{v}}$ at a vertex $\textbf{v}$ of $V(f)$ to be the sum over all symmetric copies $s_z(\textbf{v})$ of $\textbf{v}$ appearing in $\RR V^\RR_\vartheta(f)$ of the solid angles of the corresponding curvature cones $C_{s_z(\textbf{v})}$  {\it i.e.}, $\widetilde{\kappa_\textbf{v}}:=\sum_{s_z(\textbf{v}) \in \RR V^\RR_\vartheta(f)}\kappa(s_z(\textbf{v}))$.  Then $\int_{V^\RR_\vartheta(f)}\vert k^p\vert=\sum_{\textbf{v} \in Vert(V(f))}\widetilde{\kappa_{\textbf{v}}}$.
\end{rem}

\subsection{Elementary simplex case}

\begin{defi}\label{Def:Elementary}
Let $S \in \RR^{n+1}$ be a simplex with integer vertices and $(u_i)_{i\in\{1..n+1\}}$ be the collection of the vectors defined by edges issuing from one of its vertices. The simplex $S$ is {\bf elementary} if $(\overline{u_i})_{i\in\{1..n+1\}}$ is a basis of $(\ZZ_2)^{n+1}$ where $\overline{u_i}$ is the reduction modulo $2$ of $u_i$.  
\end{defi}

\begin{prop}\label{Prop:Hyperplane-Polyhedral-Curvature} Let $f$ be a tropical polynomial such that $\triangle_f$ is an elementary simplex and that $\E(f)=Vert(\triangle_f)$.
 Then, for any distribution of signs $\vartheta$, 
\begin{equation}\label{Eq:Hyperplane-Polyhedral-Curvature}
\int _{V^\RR_\vartheta(f)}\vert k^p \vert=\frac{\sigma_n}{2}.
\end{equation}
\end{prop}

In particular, Proposition~\ref{Prop:Hyperplane-Polyhedral-Curvature} holds for primitive real tropical hypersurfaces (those whose Newton polytope is primitive) since a primitive simplex is elementary. 
The two corollaries below follow from Proposition~\ref{Prop:Hyperplane-Polyhedral-Curvature} and Lemma~\ref{Prop:By-Def}.
\begin{cor}\label{Cor:Generic-Tropical-Curvature}Let $f$ be a generic tropical polynomial and let $\textbf{v}_1,\dots,\textbf{v}_l$ be the vertices of $V(f)$.  If $\triangle_{In_{\textbf{v}_i}f}$ is an elementary simplex for all $i$ then, for any distributions of signs $\vartheta$,
 $$\int_{V^\RR_\vartheta(f)}\vert k^p\vert=l \frac{\sigma_n}{2}.$$
\end{cor}
\begin{cor}\label{Cor:Primitive-Tropical-Curvature}Let $f$ be a non-singular tropical polynomial. Then, for any distribution of signs $\vartheta$, $$\int_{V^\RR_\vartheta(f)}\vert k^p\vert=(n+1)!Vol(\triangle_f)\frac{\sigma_n}{2}.$$
\end{cor}

From Proposition~\ref{Prop:Hyperplane-Polyhedral-Curvature} and Proposition~\ref{totcurvR} we deduce the equality between the polyhedral curvature and real total curvature of a non-singular real tropical hypersurface:

\begin{prop}\label{Thm:Polhedral=Limit}Let $f$ be a non-singular tropical polynomial. Then, for any distribution of signs $\vartheta$, $$\int_{V^\RR_\vartheta(f)}\vert k\vert=\int_{V^\RR_\vartheta(f)}\vert k^p\vert.$$
\end{prop}

\begin{proof} Both can be expressed as $(n+1)!Vol(\triangle_f)$ times the volume of $\RR\PP^{n}$ (see Corollary~\ref{Cor:Primitive-Tropical-Curvature} in the polyhedral case). 
\end{proof}

\subsection{Proof of Proposition ~\ref{Prop:Hyperplane-Polyhedral-Curvature}}

We will use the following proposition which essentially follows from Itenberg's Prop 3.1 in \cite{It97}.

Recall that $\RR V^\RR_\vartheta (f)= \RR V^\RR_{-\vartheta} (f)$ so for our study we might as well consider sign distributions up to total inversion of signs. We will denote $\mathcal D_{\E(f)}$ the set of sign distributions on $\E(f)$ up to simultaneous change of all signs.

\begin{prop}\label{Prop:Transitivity}
Let $S$ be an elementary simplex. The group $(\ZZ_2)^{n+1}$ acts transitively on $\mathcal D_{Vert(S)}$ via the maps $S_z$ (see Subsection~\ref{Subsec:Real-Tropical-Hypersurfaces}).
\end{prop}

\begin{proof}
In notation of Definition~\ref{Def:Elementary} the $\overline{u_i}$'s form a basis of $(\ZZ_2)^{n+1}$. Let $(\overline{z_i})_{i=1..n+1}$ be the dual basis.
 Then for each vertex $v_i$ of $S$,

\begin{itemize}
\item either $S_{\overline{z_i}}(\vartheta)(v_i)=- \vartheta(v_i)$ and $S_{\overline{z_i}}(\vartheta)(v_j)= \vartheta(v_j)$ for $v_j \neq v_i$  
\item or $S_{\overline{z_i}}(\vartheta)(v_i)= \vartheta(v_i)$ and $S_{\overline{z_i}}(\vartheta)(v_j)= -\vartheta(v_j)$ for $v_j \neq v_i$.
\end{itemize}

\end{proof}

We will prove that the curvature cones defined by a
real tropical hypersurface dual to an elementary simplex $S$ give
rise to a partition of a half-space, which yields the result.

Since we are considering the sign distributions up to total inversion of signs, we can assume that for one vertex $v_0$ of $S$, $S_z(\vartheta)(v_0)=-1$ for all $z \in {\ZZ_2}^{n+1}$.
By Proposition~\ref{Prop:Transitivity} and Remark~\ref{Rem:Curvature-cone}, the cones we need to consider are exactly those corresponding to all distributions of signs $\varphi$ on $Vert(S)$ such that $v_0$ carries a minus sign. Let us denote by $C_\varphi$ the curvature cone corresponding to such a distribution $\varphi$.

Let us prove that the curvature cones $C_\varphi$ naturally define a fan 
 which covers a half-space.

To each vertex $v$ of $S$ one associates its opposite facet $F_v$, the  
 vectorial hyperplane $H_v$ parallel to $F_v$
and a vector $n_v$ normal to $F_v$ pointing from $F_v$ to the interior of $S$. 

For a vertex $v$ of $S$ let $H_{v}^-$ and $H_{v}^+$ be respectively the half-space defined by $\{x \in \RR^{n+1}, n_{v}\cdot x \le 0\}$ and $\{x \in \RR^{n+1}, n_{v}\cdot x \ge 0\}$. 
The key point is the following fact.

\begin{lemma}\label{Lem:ConesAndHyperplanes0}
For a sign distribution $\varphi$ (with $\varphi (v_0)=-1$), the curvature cone $C_\varphi$ is the intersection of the half-spaces $H_{v}^{\varphi(v)}$ for all $v \in Vert(S)$. Moreover it is enough to intersect only those $H_{v}^{\varphi(v)}$ such that $\varphi(Vert(F_v))=\ZZ_2$.
\end{lemma}

\begin{proof}
We need to prove that the cone $C_\varphi$ is defined by  $\{x \in \RR^{n+1}, \forall v \in Vert(S),\> \varphi(v)\, n_v\cdot x \ge 0\}$.

\vspace*{1ex}
Let us denote by $E=\{e_i\}$ the set of edges of $S$ whose vertices have different signs and by $W = \{w_i\}$ the set of vectors such that $w_i$ is supported by $e_i$ and oriented from "-" to "+".
For any $v\in Vert(S)$, $W \subset H_{v}^{\varphi(v)}$. If $e_i$ is not in the face $F_v$ of $S$ opposite to $v$, it points towards the (affine) half-space containing $S$  if $\varphi(v)$ is "+" and towards the other half-space determined by $F_v$ if $\varphi(v)$ is "-". Thus $C_\varphi \subset \cap_{v\in S}  H_{v}^{\varphi(v)}$.
\vspace*{1ex}

Let us prove that each facet of $C_\varphi$ is parallel to a facet of $S$ and thus included in some $H_v$ which leads to the equality in the above inclusion.
\vspace*{1ex}

Each facet of $C_\varphi$ is a cone generated by a subset $Y$ of $W$ whose linear span $<Y>$ is of dimension $n$ and such that all vectors of $W\setminus Y$ are on the same side of $<Y>$.

 Let $Y$ be a subset of $W$, $E_Y$ be the corresponding subset of $E$ and assume that $\dim <Y> =n$ and $<Y>$ is included in no $H_v$. 
 Let us prove that the cone generated by $Y$ is not a face of  $C_\varphi$.
Each vertex of $S$ belongs to an edge of $E_Y$ (otherwise $<Y>$ would be parallel to a facet of $S$).

 Let $Vert(S)^+$ (resp. $Vert(S)^-$) be the set of vertices of $S$ with "+" respectively "-" signs.
 If either $\# Vert(S)^+$ or $\# Vert(S)^-$ is $1$ then $C_\varphi$ is just the cone of apex a vertex $v$ over its opposite face $F_v$ and its facets are cones on facets of $F_v$.
 
 Let us then assume that $\# Vert(S)^+ \ge 2$ and $\# Vert(S)^-\ge 2$. 
 If every pair of vertices in  $Vert(S)$ were connected by a chain of edges in $E_Y$ then every vertex would be connected by a chain of edge in $E_Y$ thus $\dim <Y> = n+1$ which would contradict the hypothesis.

Thus the edges in $E_Y$ split in several connected components $E_Y^i$. Each one contains at least one element of $ Vert(S)^-$ and one element of $Vert(S)^+$. The affine span $\Aff(E_Y^i)$ of $E_Y^i$ is just the affine span of the vertices it contains, thus it is the affine span of the corresponding face of $S$.

Each $\Aff(E_Y^i)$ is parallel to $<Y>$ which is of codimension one. Since $E_Y$ covers all vertices of $S$, all $\Aff(E_Y^i)$ are not in the same affine hyperplane parallel to $<Y>$. Then an affine hyperplane parallel to $<Y>$ separates the affine spans $\Aff(E_Y^1)$ and $\Aff(E_Y^2)$ of two connected components $E_Y^1$ and $E_Y^2$.

Let us pick two vertices  with minus sign ${v_1}^-$ and ${v_2}^-$ respectively in $E_Y^1$ and $E_Y^2$. 
Let $w_1$ (resp. $w_2$) be vectors in $W\setminus Y$ having origin $v_1^-$ (resp $v_2^-$) and extremity a vertex in $E_Y^2$  (resp. in $E_Y^1$).
The connected components  $E_Y^1$ and $E_Y^2$ being separated by an hyperplane parallel to $<Y>$, the vectors $w_1$ and $w_2$ are not on the same side of $<Y>$ and $Y$ does not generate a facet of  $C_\varphi$.
\vspace*{1ex}

Thus each facet of $C_\varphi$ is contained in one of the hyperplanes $H_v$ and, since we already have that $C_\varphi\subset \cap_{v\in S}  H_{v}^{\varphi(v)}$, then $C_\varphi=\cap_{v\in S}  H_{v}^{\varphi(v)}$.

\vspace*{1ex}

It is enough to intersect only those $H_{v}^{\varphi(v)}$ such that $\varphi(Vert(F_v))=\ZZ_2$. Indeed when the signs of all the vertices of a facet are the same, the intersection of $C_\varphi$  with $H_{v}$ is the origin.

\end{proof}

\begin{rem}
  In the proof of Lemma~\ref{Lem:ConesAndHyperplanes0} one can easily see that, if $\dim <Y> =n$ and $<Y>$ is included in no $H_v$, $E_Y$ has exactly two connected components. Indeed, the affine spans $\Aff(E_Y^i)$ are just the spans of pairwise disjoint faces of $S$.

Let $m$ be the number of connected components of  $E_Y$. Each $E_Y^i$ contains $\dim \Aff(E_Y^i) +1$ vertices and $\sum_i^m \dim \Aff(E_Y^i)=n$. But the number of vertices in $S$ is $n+2$, thus $\sum_i^m \left(\dim \Aff(E_Y^i) +1\right)=n+2$ and $m \le 2$. 
\end{rem}

\begin{lemma}\label{Lem:ConesAndHyperplanes}
Consider the union $A_{S}\> = \> \cup_{v\in Vert(S)} H_v$ of all linear hyperplanes $H_v$ and the collection of curvature cones $\mathcal C = (C_\varphi)_{\{\varphi \vert \varphi(v_0)=-1 \mbox{ and }\varphi(Vert(S))=\ZZ_2 \}}$. 
 The cones in $\mathcal C$ are precisely the maximal dimensional closed cones 
 in $H_{v_0}^-$ defined by $A_{S}$.
\end{lemma}

\begin{proof}
By Lemma~\ref{Lem:ConesAndHyperplanes0} for a sign distribution $\varphi$ (with $\varphi (v_0)=-1$), the curvature cone $C_\varphi$ is the intersection of the half-spaces $H_{v}^{\varphi(v)}$ for all $v \in Vert(S)$.

Any cone $D$ which is the closure of a connected component of the complement of $A_{S}\cap H_{v_0}^-$ in $ H_{v_0}^-$ is of the form $C_\varphi$. Indeed it is defined by a choice of a side for each  hyperplane $H_v$ {\it i.e.}, by the choice of the sign of the scalar product of vectors in the interior of $D$ with $n_v$ for each $v \in Vert(S)$. The sign distribution $\varphi$ is then given by $\varphi(v)=sign(n_v\cdot x)$ for any $x$ in the interior of $D$. Indeed, setting a sign on a vertex of $S$ amounts to choosing  on which side of $H_v$ are all vectors not in $H_v$ generating $C_\varphi$. The sign "-" corresponds to pointing from $F_v$ to the exterior of $S$ and "+" from $F_v$ to the interior of $S$.
(Of course cones defined by a choice of side for each $n+2$ hyperplanes $H_v$ can sometimes be reduced
 to the origin; this corresponds exactly to a sign distribution $\varphi$ on $Vert(S)$ which does not surjects on $\ZZ_2$, {\it i.e.}, to an empty orthant on the real tropical variety side.)
\end{proof}

Thus the closed cones $C_\varphi$ clearly cover $H_{v_0}^-$. (A vector $x$ in $H_{v_0}^-$ not belonging to one of the $H_v$'s is in $C_\varphi$ if and only if for all $v \in Vert(S)$, $\varphi(v)=sign(n_v\cdot x)$.) Moreover by Lemma~\ref{Lem:ConesAndHyperplanes} they realise a subdivision of $H_{v_0}^-$. So, since by Proposition~\ref{Prop:Transitivity} we get all possible sign distributions such that $\varphi (v_0)=-1$, we proved that $\int _{V^\RR_\vartheta(f)}\vert k^p \vert=\frac{vol(S^n)}{2}=\frac{\sigma_n}{2}$.
\qed
\begin{exa}\label{Ex:Total-Curvature-Real-Line}
On Figure~\ref{Fig:example} and Figure~\ref{Fig:Line-Total-Curvature} we illustrate Proposition~\ref{Prop:Hyperplane-Polyhedral-Curvature} on the trivial case of a real tropical line. We depict the angles formed by the vectors generating the curvature cones which in this case are the angles of the Newton triangle. 
\end{exa}

\begin{figure} [htbp]
\begin{center}
\resizebox{\textwidth}{!}{\input{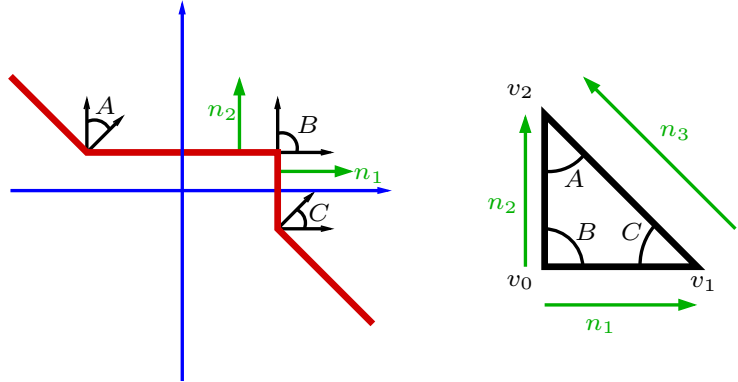}}
\caption {In the case of a tropical curve ($n=1$), Proposition~\ref{Prop:Hyperplane-Polyhedral-Curvature} follows from the fact that the sum of the interior angles of a triangle is equal to $\pi$.}
\label{Fig:example}
\end{center}
\end{figure}

\begin{figure} [htbp]
\begin{center}
\resizebox{\textwidth}{!}{\input{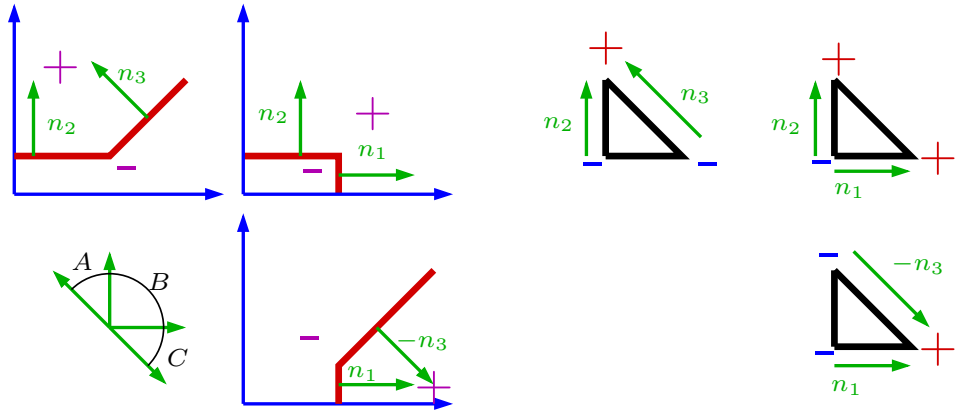}}
\caption{The curvature cones for the real line cover  $H_{v_0}^-$. The symmetric copies of the non-empty quadrants are shown on the left and the corresponding sign distributions at the vertices of the Newton polygon on the right. The bottom left picture shows how the curvature cones cover a half-plane.}
\label{Fig:Line-Total-Curvature}
\end{center}
\end{figure}


\section{Complement}\label{Sec:Complement}


\subsection{Tropical lower bound}\label{Subsec:T-Varieties}

In a forthcoming paper, the second author studies, using tropical geometry, the limit when $t$ goes to zero of the real total curvature of a family $\RR V(F_t)$ of real algebraic hypersurfaces. This study allows  to give a lower bound to this limit, for any polynomial $F\in\KK_\RR[x]$ realising a generic tropical hypersurface. 
This bound depends only on the tropicalisation of $F$ and will be called the \textit{tropical bound}.\\

Then, if $k$ is the classical curvature function, $$\lim_{t\rightarrow 0}\int_{\RR V(F_t)}\vert k\vert dv$$ 
will be bounded from above by the Risler's complex bound (see Inequation~(\ref{ineqcurv})) and from below by the tropical bound.  A polynomial $F\in\KK_\RR[x]$ is called \textit{maximal} with respect to the real total curvature if the Risler's upper bound is sharp. In other words, a polynomial $F\in\KK_\RR[x]$ is maximal with respect to the real total curvature if $$\lim_{t\rightarrow 0}\int_{\RR V(F_t)}\vert k\vert dv=\frac{\sigma_n}{\sigma_{2n}}\int_{V(F_t)}\vert K(x(t))\vert dv.$$

The idea behind the construction of the tropical bound is to look for the valuation of points in $V(F)\cap (\KK_\RR^*)^{n+1}$ that \textit{concentrate} the real total curvature: a point $x\in V(F)\cap (\KK_\RR^*)^{n+1}$ concentrates the real total curvature if for any family of neighbourhoods $\{U_t\}_{0<t\ll 1}$ of the family of points $x(t)$, $$\lim_{t\rightarrow 0}\int_{\RR V(F_t)\cap U_t}\vert k\vert dv\geq \frac{\sigma_n}{2}.$$

Via tropical methods, the second author gives a lower bound for the number of points in $V(F)$ that concentrates the real total curvature and the tropical bound arises as a direct consequence. Using this study, an infinite family of polynomials in $\KK_\RR[x]$ whose tropical bound is equal to their Risler's complex bound is constructed. This is a tropical proof of the following theorem:\\

\begin{thm}\label{Thm:Lucia}
For any $d\in\NN$ and any $n\in\NN$, there exists real polynomials of degree $d$ in $\KK[x_1,\dots,x_{n+1}]$ 
 maximal
 with respect to the real total curvature.
\end{thm} 

One deduces from this theorem a tropical proof of Orevkov's observation (see \cite{orevkov}) about the sharpness (up to any $\epsilon>0$) of Risler's complex bound for affine real algebraic hypersurfaces.
 In the Viro's patchworking language, this result has been also proved in \cite{lopezdemedrano}.


\subsection{Gauss-Bonnet }\label{Subsec:Gauss-Bonnet}
Let $f$ be a non-singular tropical polynomial with Newton polytope $\triangle_f$, $V(f)$ the tropical variety it defines and $\triangle_f = \cup \triangle_i$ be its dual (primitive) triangulation.\\
 \begin{prop}
 Let $V(f)$ be a non-singular tropical hypersurface with Newton polytope $\triangle_f$, and $V \subset (\CC^*)^{n+1}$ be a generic complex hypersurface with Newton polyhedra $\triangle_f$. Then:
 \begin{equation} \label{G-B}
  \int_{V(f)} K =  a_n \frac{\sigma_{2n+1}}{\sigma_1} \chi (V)
  \end{equation}
 where $\chi(V)$ stands for the {\it Euler characteristic} of $V$.
 \end{prop}
 {\bf Remarks}
  \par a) If $M$ is a compact real variety of even dimension $n = 2m$, the classical Gauss-Bonnet formula is
 \[
 \int_{M} k= (-1)^m \frac{\sigma_n}{2}  \chi(M).
 \]

 \par b) For $n = 1$ (curves), (\ref{G-B}) gives:
 \[   \int_{V(f)} K = 2\pi \chi(V) = - 4 \pi vol(\triangle_f).   \]
 
 \par c) For $n = 2$ (surfaces), (\ref{G-B}) gives:
 \[\int_{V(f)} K  = 4 \frac{\pi^2}{ 3} \chi (V) = 8 \pi^2 vol(\triangle_f).
 \]
 \begin{proof} (of (\ref{G-B})). \\If $V \subset (\CC^*)^{n+1}$ is a generic hypersurface with Newton polytope $\triangle_f$,\break one has $\chi(V)= (-1)^n  (n+1)! \, vol(\triangle_f)$ (Hovansky's formula, cf. \cite{Hov78}). The tropical variety $V(f)$ is by hypothesis dual to a primitive triangulation $\triangle_f = \cup \triangle_i$ with $vol \triangle_i = 1/ (n+1)!$. \\
 Let $r$ be the number of $\triangle_i$'s ({\it i.e.}, the number of vertices of $V(f)$); then one has:
 \[\int_{V(f)} K = r (-1)^n a_n vol( \CC \PP^n) = r (-1)^n a_n \frac{\sigma_{2n + 1}}{\sigma_1} \]
  by Definition \ref{Def:Complex-Tropical-Curvature}.\\
This proves (\ref{G-B}), because $vol(\triangle_f) = r \times (1/(n+1)!)$, and then $(-1)^n r = \chi(V)$.
\end{proof}

\def\cprime{$'$} \def\cprime{$'$}
\providecommand{\bysame}{\leavevmode\hbox to3em{\hrulefill}\thinspace}
\providecommand{\MR}{\relax\ifhmode\unskip\space\fi MR }
\providecommand{\MRhref}[2]{%
  \href{http://www.ams.org/mathscinet-getitem?mr=#1}{#2}
}
\providecommand{\href}[2]{#2}

\end{document}